\documentclass[12pt, oneside]{article}   	
\usepackage[left=2.5cm, right=2.5cm, bottom=2.5cm]{geometry}

\usepackage{graphicx}				

\usepackage{xcolor}
\usepackage{indentfirst}
\usepackage{mathrsfs}
\usepackage{ulem}	
\usepackage{amsthm}
\usepackage{amsfonts}
\usepackage{amsmath}
\usepackage{amscd}
\usepackage{tikz}
\usepackage{enumerate}
\usepackage{amssymb}	
\usepackage{verbatim}
\usepackage{cancel}
\usepackage{cite}
\usepackage{textcomp}
\usepackage{hyperref}

\usetikzlibrary{decorations.pathreplacing,decorations.markings}
\tikzset{
  on each segment/.style={
    decorate,
    decoration={
      show path construction,
      moveto code={},
      lineto code={
        \path [#1]
        (\tikzinputsegmentfirst) -- (\tikzinputsegmentlast);
      },
      curveto code={
        \path [#1] (\tikzinputsegmentfirst)
        .. controls
        (\tikzinputsegmentsupporta) and (\tikzinputsegmentsupportb)
        ..
        (\tikzinputsegmentlast);
      },
      closepath code={
        \path [#1]
        (\tikzinputsegmentfirst) -- (\tikzinputsegmentlast);
      },
    },
  },
  mid arrow/.style={postaction={decorate,decoration={
        markings,
        mark=at position .5 with {\arrow[#1]{stealth}}
      }}},
}

\newtheorem{theorem}{Theorem}[section]
\newtheorem{definition}{Definition}[section]
\newtheorem{proposition}{Proposition}[section]
\newtheorem{lemma}{Lemma}[section]
\newtheorem{remark}{Remark}[section]

\def\bl{\begin{lemma}}
\def\el{\end{lemma}}
\def\cC{\mathcal C}

\def\C{\mathbb{C}}
\def\R{\mathbb{R}}

\def\hnu{{\hat\nu}}
\def\hgamma{{\hat\gamma}}

\def\vp{\varphi}
\def\be{\begin{equation}}
\def\ee{\end{equation}}
\def\DS {\displaystyle}

\title{On the ergodic properties of time changes of partially hyperbolic homogeneous flows}
\author{Changguang Dong}
\date{\today}

\begin{document}
\maketitle
\begin{abstract}
For any accessible partially hyperbolic homogeneous flow, we show that all smooth time changes are K and hence mixing of all orders. We also establish stable ergodicity for time-one map of these time changes.
\end{abstract}

\tableofcontents

\section{Introduction and results}

This paper is a contribution to our understanding of a natural class of partially hyperbolic systems. Namely we will be interested in reparametrizations (or time changes) of algebraic flows on homogeneous spaces. Algebraic flows, roughly speaking, contain three categories: elliptic, parabolic and 
(partially) hyperbolic. These are determined by the divergence rate of nearby orbits. For these flows, the reparametrizations (time changes) could reflect different ergodic properties. 

In the literature, time changes are well studied in the case of elliptic flows. Time changes often don't change the dynamics like for rotations or linear flows \cite{Ko} but sometimes lead to mixing phenomenon \cite{Fa1,Fa2}. For parabolic systems, there has been recently a lot attention in studying time changes. These include time changes of horocycle flow \cite{R-h,KLU,FU,FF, RT}, Heisenberg nilflows \cite{AFU,FK1,FK2}, general unipotent flows \cite{Si}, and general nilflows \cite{afru,Rav}. 

However, for partially hyperbolic systems, time changes are not well understood. Here, we are specifically interested in the ergodic properties of (partially) hyperbolic homogeneous flows. The first study of this kind dates back to Hopf \cite{ho}, who proved the ergodicity of geodesic flow on surfaces of constant negative curvature. Later, they are proven to be Bernoulli \cite{ow}. Ratner \cite{ra} showed a much more general result: Anosov flow with a Gibbs measure is Bernoulli. Dani proved the K property for partially hyperbolic homogeneous flows \cite{da}, and the Bernoulli property for a special class of these flows \cite{da-1}. 
Small time changes of homogeneous flows are K by the results of
\cite{PSST}.
Very recently, a remarkable result by Kanigowski \cite{Ka}, asserts that all mixing partially hyperbolic homogeneous flows are Bernoulli.

 It is natural to ask, whether a certain ergodic property persists among smooth time changes of these partially hyperbolic homogeneous flows. This is a motivation of our current work. We can show that for any partially hyperbolic homogeneous flows with accessibility, all smooth time changes admit K property, and hence are mixing. 
 This means, in particular, K property (and mixing property) persists under time change.

To be more precise, let $G$ be a semisimple Lie group over $\R$ or $\C$, and $\Gamma<G$ be a cocompact lattice. Denote $M:=G/\Gamma$. Let $g_t$ be a flow on $M$ generated by the vector field $X$, meaning that $g_t(x)=\exp(tX)x$ for any $x\in M$. The flow preserves the Haar measure $m$. We assume the entropy of $g_t$ is positive. Hence $g_t $ is a partially hyperbolic flow. Let $\tau:M\to\mathbb R$ be a positive smooth
function. The time change $g^\tau_t$ is a flow generated by the vector field $(\tau)^{-1} X$. 
In other words, we have the new flow $g^\tau_t(x):=g_{\alpha(x,t)}(x)$, where $\alpha(x,t)$ is given by $$\int_0^{\alpha(x,t)}\tau(g_s(x))ds=t.$$ 
Some useful properties of $g_t^\tau$ are known. The flow $g^\tau_t$ preserves a smooth measure $m^\tau$ given by $dm^{\tau}=\frac{\tau}{\int_M\tau dm}dm$, and the entropy $h_{m^\tau}(g_t^\tau)>0$. The new flow is ergodic if $g_t$ is. 

Our first result is

\begin{theorem}\label{main}
Let $G$ be a semisimple Lie group over $\R$ or $\C$, and $\Gamma<G$ be a cocompact lattice. Let $g_t$ be a flow on $G/\Gamma$ generated by a vector field $X\in$ Lie($G$). Assume that $g_t$ is partially hyperbolic, and accessible. Let $\tau:G/\Gamma\to\mathbb R$ be a smooth positive
function, and $g^\tau_t$ be the time change generated by $\tau^{-1}X$. 
Then $g^\tau_t$ is accessible. Moreover, $g^\tau_t$ has 
K property, hence it is mixing of all orders.
\end{theorem}

Here the accessibility condition implies that the group $G$ has no compact factors \cite{ps}.
Let's remark that there are examples of partially hyperbolic diffeomorphisms which are K but the mixing 
(and hence, K property) can be destroyed under certain time change. For example, consider a constant suspension flow over a hyperbolic automorphism on $\mathbb T^2$, for any nonzero $t$, the time-$t$ map is never mixing with respect to the volume measure, however one can find a nontrivial time change so that the new flow is accessible, and hence mixing (in fact Bernoulli).

Our strategy of proving K property, is to show that the flow is center bunched and also accessible. 

Firstly, we verify the partial hyperbolicity and center bunching, which is proven by computing the push-forwards of any vector in the tangent bundle.
\begin{theorem}\label{m-c-b}
Let $G$ be a semisimple Lie group over $\R$ or $\C$, and $\Gamma<G$ be a cocompact lattice. Let $g_t$ be a flow on $G/\Gamma$ generated by a vector field $X\in$ Lie($G$). Assume that $g_t$ is partially hyperbolic. Let $\tau:G/\Gamma\to\mathbb R$ be a smooth positive
function, and $g^\tau_t$ be the time change generated by $\tau^{-1}X$. 
Then $g_t^\tau$ is 
partially hyperbolic and center bunched
\footnote{ The relevant facts from the theory of partially hyperbolic systems will be recalled
in \S \ref{SSPhD}.}.
\end{theorem}

Then we build a correspondence between the stable/unstable manifolds of $g_t$ and those of $g^\tau_t$. By studying how the accessibility classes change under time change, we can  show:

\begin{theorem}\label{main-a}
Under all the assumptions of Theorem \ref{main}, the time change $g^\tau_t$ is accessible.
\end{theorem}

By Theorem \ref{m-c-b} and \ref{main-a}, and a celebrated result of Burns-Wilkinson \cite{BW} (Theorem \ref{bw}), we thus obtain the K property.

We can push Theorem \ref{main-a} forward, and obtain a stronger result, that the time change is stably accessible. This is enough (by Theorem \ref{bw}) to conclude that the time change is stably ergodic. More precisely, we have

\begin{theorem}\label{main-stable}
Under all the assumptions of Theorem \ref{main}, the time-$t_0$ map $g^\tau_{t_0}$ is stably ergodic (and in fact, stably K) for any $t_0\neq 0$.
\end{theorem}

Let us mention that, 
stable ergodicity is well studied for homogeneous flows or algebraic maps, for certain toral automorphism \cite{fe}, for geodesic flows, or compact group extensions of geodesic flow \cite{BW-1,bm}, and for general affine maps on homogeneous spaces \cite{ps}. Theorem \ref{main-stable} can be considered a generalization of some results in \cite{ps}. It also shows that for these homogeneous flows, the property of 
stable ergodicity persists under smooth time change. 

The accessibility assumption in Theorem \ref{main} and \ref{main-stable} is technical. If the flow $g_t$ is not accessible but still ergodic, then it is essentially accessible. When the orbit foliation is a sub-foliation of the accessibility class of $g_t$, then arguing similarly to our proof, we can obtain the essential accessibility of the time change (hence the K property). However our argument fails in other cases, in which the accessibility class does not necessarily remain the same after the time change.  Nevertheless, we believe that the ergodic properties (such as K property, Bernoulli property, and stable ergodicity) persist under smooth time  changes.

The structure of this paper is as follows. In \S \ref{S-P}, we present some preliminaries including a few relevant definitions. In \S \ref{S-PC}, we show partial hyperbolicity of the time change and complete the proof of Theorem \ref{m-c-b}. Here our result applies to a much more general setting. In \S \ref{S-PA}, we obtain some preparatory results and give some sufficient conditions of accessibility. In \S \ref{S-A}, we prove Theorem \ref{main-a} and hence Theorem \ref{main}, where the homogeneous structure plays an important role (see the proof of Theorem \ref{new}). We can also obtain some results by using the property of fundamental group of the manifold. In \S \ref{S-C}, we verify centre engulfing of the system, thus prove Theorem \ref{main-stable}.

\textbf{Acknowledgement:} The author is deeply indebted to D. Dolgopyat for his careful read and useful comments. He also thanks G. Forni and A. Kanigowski for their various help during the preparation of this work.

\section{Preliminaries}
\label{S-P}

\subsection{Partial hyperbolicity}
\label{SSPhD}

Let $f$ be a diffeomorphism of a compact manifold $M$, $f$ is called {\em pointwise partially hyperbolic} if there is 
a $Df$ invariant splitting 
$$TM=E^s\oplus E^c \oplus E^u $$
and there exist a Riemannian metric and positive continuous functions
$\nu, \hnu, \gamma, \hgamma$ such that 
$$ \nu, \hnu<1, \quad \nu<\gamma<\hgamma^{-1}<\nu^{-1} $$
and
$$ \|Df {\bf v}\|\leq \nu(x)\|{\bf v}\| \text{ for } {\bf v}\in E^s(x) ;$$
$$ \gamma(x)\|{\bf v}\|\leq \|Df {\bf v}\|\leq \hgamma^{-1} (x)\|{\bf v}\| \text{ for } {\bf v}\in E^c(x) ;$$
$$ \hnu^{-1} (x)\|{\bf v}\|\leq \|Df {\bf v}\|  \text{ for } {\bf v}\in E^u(x). $$
$f$  is called {\em pointwise center bunched} if additioanlly
$$ \nu<\gamma \hgamma, \quad \hnu<\gamma \hgamma. $$

Let $g_t$ be a flow a compact manifold $M.$ $g$ is called {\em uniformly partially hyperbolic} if there is 
a $Dg_t$ invariant splitting 
$$TM=E^s\oplus E^c \oplus E^u $$
and there exist a Riemannian metric and positive numbers
$a, b, A, B$ and $C$ such that 
$$ a<A, \quad b<B$$ 
and for all $t>0$,
$$ \|Dg_t {\bf v}\|\leq Ce^{-Bt} \|{\bf v}\| \text{ for } {\bf v}\in E^s ;$$
$$ C^{-1}e^{-at} \|{\bf v}\|\leq \|Dg_t {\bf v}\|\leq Ce^{bt} \|{\bf v}\| \text{ for } v\in E^c ;$$
$$C^{-1} e^{At} \|{\bf v}\|\leq \|Dg_t {\bf v}\|  \text{ for } {\bf v}\in E^u. $$
$g_t$  is called {\em uniformly center bunched} if additioanlly
$$ a+b<A, \quad a+b<B. $$

It is well known that the distributions $E^s$ and $E^u$ are integrable. 
We  let $\mathcal F^s$ and $\mathcal F^u$ be the corresponding foliations.
In general the distribution $E^c$ is not integrable.

\begin{definition}
A partially hyperbolic diffeomorphism $f$ is dynamically coherent if the distributions $E^c$, $E^s\oplus E^c$ and $E^c\oplus E^u$ are uniquely integrable.
\end{definition}

An $su$-path is a concatenation of finitely many subpaths, each of which lies entirely in a single leaf of $\mathcal F^s$ or a single leaf of $\mathcal F^u$. An $su$-cycle is an $su$-path beginning and ending at the same point.

\begin{definition}
We say that a partially hyperbolic diffeomorphism $f:M\to M$ is accessible if any point in M can be connected from any other point via an $su$-path. 
\end{definition}

The accessibility class $AC(x)$ of $x\in M$ is the set of all $y\in M$ that can be reached from $x$ via an $su$-path. Accessibility means that there is one accessibility class, which contains all points. A weaker notion is essential accessibility defined below.

\begin{definition}
We say that a partially hyperbolic diffeomorphism $f:M\to M$ is essentially accessible if for a volume measure, any measurable set that is a union of entire accessibility classes has either full measure or zero measure.
\end{definition}

We may also need the notion of local accessibility.

\begin{definition}\label{def-local-A}
A partially hyperbolic diffeomorphism is said to be locally accessible if for any $\epsilon>0$, there is $\delta=\delta(\epsilon)$ with $\delta\to 0$ as $\epsilon\to 0$, such that for any $x,y$ with $d(x,y)<\epsilon$, there is an $su$-path in the $\delta$
neighborhood
 of $x$ that connects $x$ and $y$.
\end{definition}

\subsection{Ergodicity and K property}

\begin{definition}
We say, $(T,X,\mu)$ is K (or has the K property) if the $\sigma$-algebra generated by $\{A:h_\mu(T,\{A,X\backslash A\})=0\}$ is equal to $\{\emptyset, X\}$ $\mu$ almost everywhere.
\end{definition}

By Rokhlin-Sinai theorem \cite{RS}, the above is equivalent to the system not 
having non-trivial zero entropy factors. It is a well-known fact that K implies mixing of all orders.

Our strategy to prove the K property is to show accessibility and then apply the following celebrated result.

\begin{theorem}[\cite{BW}]\label{bw}
Suppose a diffeomorphism $f$ is $C^2$, volume preserving, pointwise partially hyperbolic and pointwise center bunched. If $f$ is essentially accessible, then $f$ is ergodic, and in fact has the K property.
\end{theorem}

\begin{definition}
We say that, a diffeomorphism $f\in \text{Diff}^2(M,\mu)$ is stably ergodic (K), if there is an open neighborhood $\mathcal U\subset \text{Diff}^2(M,\mu)$ of $f$, such that every $g\in \mathcal U$ is ergodic (K) w.r.t. $\mu$.
\end{definition}

\subsection{Time change of a flow}

Let $(\phi_t,S,\mu)$ be a measurable flow, $\tau:S\to \mathbb R$ be a positive function. The time-changed flow $\phi_t^\tau$ is defined by $\phi_t^\tau(x):=\phi_{\alpha(x,t)}(x)$, where $\alpha(x,t)$ satisfies $$\int^{\alpha(x,t)}_0\tau(\phi_s(x))d s=t.$$ The time-changed flow preserves a new measure $\mu^\tau$ satisfying $\frac{d\mu^\tau}{d\mu}=\frac{\tau}{\int\tau d\mu}$. It is easy to check that $\alpha(x,t)$ is a cocycle over $\phi_t^\tau$, namely for all $s,t$ and almost $x\in S$, $$\alpha(x,t+s)=\alpha(x,s)+\alpha(\phi^\tau_s(x),t).$$ Let $v(x,t):=\int^t_0\tau(\phi_s(x))ds$, then $v$ is a cocycle over $\phi_s$, and $$v(x,\alpha(x,t))=\alpha(x,v(x,t))=t.$$ Suppose now we have another time-changed flow $\phi_t^\kappa$ with $w(x,t):=\int^t_0\kappa(\phi_s(x))ds$, if $v(x,t)-w(x,t)=\xi(x)-\xi(\phi_t(x))$ for a measurable function $\xi$, then $\phi_t^\tau$ is conjugate to $\phi_t^\kappa$ via the map $h:x\mapsto \phi_{\xi(x)}^\kappa(x)$.

If $\phi_t$ is a smooth flow generated by a vector field $  X$ and $\tau$ is a smooth function, then $\phi_t^\tau$ is a smooth flow generated by $\tau^{-1} X$.

\section{Partial hyperbolicity and center bunching of the time change}
\label{S-PC}
In this section, we will show that the time change is partially hyperbolic and center bunched in a much more general setting. 

In order to verify partial hyperbolicity, we need to find an invariant splitting first. It is not easy to obtain the splitting directly by using algebraic methods. 
Nevertheless, we have a very nice observation that the new stable (respectively unstable) manifold is in fact a graph of the original stable (respectively unstable) manifold towards the flow direction. This idea enables us to construct stable and unstable manifolds, and thus obtain an invariant splitting, which is a priori only H\"older continuous. The low regularity of the splitting is also a reason that we could not succeed by applying algebraic method such as Lie bracket.

Once we obtain the splitting, we compute the push-forwards of the vector fields by the time change. Then we can show pointwise partial hyperbolicity, as well as center bunching. 

We remark that, the verification only for large $t$ is not a problem for us to apply Theorem \ref{bw}, because stable ergodicity (and stably K) established for large $t$ will also hold for small $t$.

\subsection{New stable and unstable manifolds for time changes}

Assume that $g_t$ is a partially hyperbolic flow on a compact manifold $M$, and $\tau:M\to \R$ is a positive smooth function. Let $g_t^\tau$ be the time change.

Let $W^{s/u}$ be the stable (unstable) manifold for $g_t$. The following important observation reflects a closed connection between the old stable/unstable manifold and the new stable/unstable manifold.

\begin{proposition}\label{inv-p}
The stable manifold of $g^\tau_t$ is given by a graph of a smooth map from the stable manifold of $g_t$ to the flow orbit. It is also the case for the unstable manifold.
\end{proposition}

\begin{proof}
Fix a point $x\in M$. Let $y\in W^s(x)$. Consider $$\beta^s(x,y):=\int_0^\infty(\tau(g_rx)-\tau(g_ry))dr,$$it is clear  that $\beta^s$ is finite since $\tau(g_rx)-\tau(g_ry)$ decays exponentially. 

Now let $T_x(y)$ be given by $$\int_0^{T_x(y)}\tau(g_r(y))dr=-\beta^s(x,y).$$ 
Define a map $\Phi_x^s$ by $\Phi_x^s(y):=g_{T_x(y)}(y)=g^\tau_{-\beta^s(x,y)}(y)$. Let $W^s_\tau(x):=\{g^\tau_{-\beta^s(x,y)} y\}_{y\in W^s(x)}$.

We can show that $\{W^s_\tau(x):x\in M\}$ form an invariant foliation. Moreover, for $y\in W^s_\tau(x)$, $d(g_t^\tau x, g_t^\tau y)\to 0 $ as $t\to+\infty$ exponentially fast. Notice that $$d(g_t^\tau(x),g_t^\tau(\Phi_x^s(y)))=d(g_{\alpha(x,t)}(x),g_{\alpha(y,t-\beta^s(x,y))}(y)).$$
By the definition of $\alpha$, $$\int^{\alpha(x,t)}_0\tau(g_r(x))dr=t,\;\;\int^{\alpha(y,t-\beta^s(x,y))}_0\tau(g_r(y))dr=t-\beta^s(x,y).$$Therefore $$\int_0^{\alpha(x,t)}(\tau(g_rx)-\tau(g_ry))dr+\int_{\alpha(x,t)}^{\alpha(y,t-\beta^s(x,y))}\tau(g_r(y))dr=\beta^s(x,y),$$by taking limit and the definition of $\beta^s(\cdot,\cdot)$, $$\lim_{t\to\infty}\int_{\alpha(x,t)}^{\alpha(y,t-\beta^s(x,y))}\tau(g_r(y))dr=0.$$ Since $\tau>0$, $$\lim_{t\to\infty}(\alpha(x,t)-\alpha(y,t-\beta^s(x,y)))=0,$$ it follows that $d(g_t^\tau(x),g_t^\tau(\Phi_x^s(y)))\to 0,$ as $t\to\infty$. In fact, we can have $d(g_t^\tau x,g_t^\tau\Phi_x^s(y))$ decays exponentially for $y\in W^s(x)$. This is because by slightly modifying the above calculation, there exist positive constants $c,C$ such that for $t>0$ $$|\alpha(x,t)-\alpha(y,t-\beta^s(x,y))|\le Ce^{-ct},$$thus the exponential contraction follows easily by triangle inequality.

\begin{lemma}
$W^s_\tau(x)$ is a smooth sub-manifold.
\end{lemma}

\begin{proof}
Fix $x$, we need to check that the map $y\mapsto \beta^s(x,y)$ is smooth for $y\in W^s(x)$. However,
for ${\bf v}\in E^s(y)=T_y W^s(x)$ we have
$$ \partial_{\bf v} \beta^s(x,y)=-\lim_{T\to\infty} \partial_{\bf v} \tau_T(y)=-\lim_{T\to\infty}\int_0^T \partial_{D g_t {\bf v}}\tau(g_t y) dt$$
and the integrand decays exponentially since $\DS \|Dg_t {\bf v}\|\leq Ce^{-B t}\|{\bf v}\|. $
\end{proof}

The proof for the unstable manifold is similar. We omit the proof here, but state the expression in the following. Let $\beta^u(x,y):=-\int_0^\infty(\tau(g_{-r}x)-\tau(g_{-r}y))dr$. Define a map $\Phi_x^u$ by $\Phi_x^u(y):=g^\tau_{-\beta^u(x,y)}(y)$. We can show $\Phi_x^u(y)\in W^u_\tau(x)$ analogously.
\end{proof}

\begin{remark}
If $y\in W^s(x)$ and $y'\in W^s_\tau(x)$ with $y'=\Phi_x^s(y)$, then Proposition \ref{inv-p} gives that $\int_0^\infty (\tau(g_rx)-\tau(g_ry))dr=-\int_0^{T_x(y)}\tau(g_ry)dr$, which roughly speaking, is equivalent to saying that the integral of $\tau-\int\tau$ along the forward orbit of $x$ asymptotically equals to that along the forward orbit of $y'$. 
\end{remark}

\begin{proposition}\label{p-split}
For any $t>0$, there is a $g^\tau_t$ invariant splitting $TM=\tilde{E}^s\oplus E^c\oplus \tilde{E}^u$. More explicitly,
for each $x$, there is a linear transformation $L^s_x:E^s(x)\to \tilde{E}^s(x)$ such that $\tilde{E}^s(x)=L^s_x(E^s(x))$. Moreover, the map $x\mapsto L^s_x$ is continuous. Analogous result for $\tilde{E}^u$ is also true.
\end{proposition}
\begin{proof}
By Proposition \ref{inv-p}, it follows that the space $\tilde{E}^s=TW^s_\tau$ is invariant. By a direct computation, 
$\DS \tilde{E}^s(x)=L^s_x E^s(x)$ where $\DS L^s_x {\bf v}={\bf v}+\frac{\partial_{\bf v} \beta^s(x,\cdot)}{\tau} X$. Here, $X$ is the vector field generating the flow $g_t$. The continuity of $x\mapsto L_x^s$ follows from the continuity of $\beta^s(x, \cdot)$ on $x$.
Since by the definition of partial derivative, $|\partial_{\bf v} \beta^s(x,\cdot)|$ is proportional to $\|\bf v\|$, there is a constant $K$ such that for all $x\in M$, ${\bf v}\in E^s(x)$,
\begin{equation}
\label{Comp}
 \frac{\|{\bf v}\|}{K}\leq \|L^s_x{\bf v}\|\leq K \|{\bf v}\|. 
\end{equation}
It is analogous for $\tilde{E}^u$.
\end{proof}

\begin{proposition}
The flow $g_t^\tau$ is dynamically coherent.
\end{proposition}
\begin{proof}
It is straightforward to verify the joint integrability (and unique integrability) of $E^c$ with $\tilde{E}^s,\tilde{E}^u$. Notice also that $g_t^\tau$ shares the same center (and also center-stable, center-unstable) foliation with $g_t$. The time change is dynamically coherent since $g_t$ is.
\end{proof}

\subsection{Center bunching for time changes}

We have the following general result, which is applicable to any partially hyperbolic flow.

\begin{theorem}\label{th-ce}
Suppose $g_t$ is a uniformly partially hyperbolic flow and $\tau$ is a smooth positive function. 
Then (for large $T$), $g_T^\tau$ as a diffeomorphism, is pointwise partially hyperbolic (with the functions depending on $T$). If $g_t$ is uniformly center bunched
then for large $T$, $g_T^\tau$ is pointwise center bunched.
\end{theorem}

Note that $g_T^\tau$ preserves $E^c.$ Recall that below the constants $a,b,A,B$ are constants in the definition of uniformly partial hyperbolicity of $g_t$.

\begin{lemma}
(a) If ${\bf v}\in E^c(x)$ then $\DS \|D_xg_T^\tau  {\bf v}\|\leq C T e^{b \alpha(x, T)} \|{\bf v}\| $.

(b) If ${\bf v}\in E^c(x)$ then $\DS \|D_xg_T^\tau  {\bf v}\|\geq C^{-1} T e^{-a \alpha(x,T)} \|{\bf v}\| $.
\end{lemma}

\begin{proof}
Firstly, notice that $g^\tau_t(x)=g_{\alpha(x,t)}(x)$, then $ D(g_T^\tau) {\bf v}=Dg_T^\tau {\bf v}+\partial_{\bf v} \alpha(x, T) X,$
where $X$ is the vector field generating $g_t.$ 

Denoting $\DS \tau_T(x)=\int_0^T \tau(g_t x) dt$, 
we get from the equation
$ \tau_{\alpha(x, T)}(x)=T $
that $$\DS \partial_{\bf v} \alpha(x, T) \tau(g_{\alpha(x,T)}(x))=-\partial_{\bf v} \tau_{\alpha(x,T)}(x). $$
Since
$$ \partial_{\bf v} \tau_{\alpha(x,t)} (x)=\int_0^{\alpha(x,t)} [\partial_{Dg_t {\bf v}} \tau](g_t x) dt$$
and $\|Dg_t {\bf v}\|\leq C_0e^{bt}$, part (a) follows.

Part (b) follows from part (a) applied to $(g_T^\tau)^{-1}$ at $g_{\alpha(x, T)}(x)$.
\end{proof}

\begin{remark} 
The argument given above also shows that if the derivative in the center direction grows
not faster than $C t^l$ for some $l>0$ then the derivative in the center direction for the time change grows
by at most $\bar{C} t^{l+1}.$ This is the case for homogeneous flows.
\end{remark}

Recall that $W_\tau^s(x)=\{g^\tau_{-\beta^s(x,y)} y\}_{y\in W^s(x)}$ where
$$ \beta^s(x,y)=\lim_{T\to+\infty} [\tau_T(x)-\tau_T(y)].$$

\begin{lemma}
There is a constant $R$ such that for ${\bf u}\in \tilde{E}^s(x)$, 
$$\| Dg_T^\tau{\bf u}\|\leq R e^{-B \alpha(x, T)} \|{\bf u}\|. $$
\end{lemma}

\begin{proof}
Let ${\bf u}=L^s_x({\bf v}),$ here $L_x^s$ is from Proposition \ref{p-split}. Then $\DS D g_T^\tau {\bf u}=L^s_{g^\tau_Tx}(Dg_{\alpha(x,T)} {\bf v}).$
Now \eqref{Comp} shows that 
$\DS \frac{\|Dg_T^\tau {\bf u}\|}{\|{\bf u}\|}$  and $\DS \frac{\|Dg_{\alpha(x,T)} {\bf v}\|}{\|{\bf v}\|}$
are within factor  $K^2$ from each other, thus finishing the proof.
\end{proof}

By applying the last lemma to $(g_T^\tau)^{-1}$ we get
\begin{lemma}
$\DS \| Dg_T^\tau {\bf u}\|\geq R^{-1} e^{A \alpha(x, T)} \|{\bf u}\|$  for ${\bf u} \in \tilde{E}^u(x). $
\end{lemma}

\begin{proof}[Proof of Theorem \ref{th-ce}]
The above lemmata together show that $g_T^\tau$ is pointwise partially hyperbolic with
$$ \nu=e^{-\alpha(x,T) B+o(T)}, \quad
\hnu=e^{-\alpha(x,T) A+o(T)}, \quad
\gamma=e^{-\alpha(x,T) a+o(T)}, \quad
\hgamma=e^{-\alpha(x,T) b+o(T)}. 
$$
Now the statements of the theorem follow easily.
\end{proof}

Theorem \ref{m-c-b} is a straightforward corollary of Theorem \ref{th-ce}, as follows.

\begin{proof}[Proof of Theorem \ref{m-c-b}]
Notice that in \cite{ps}, it is shown that $g_{t}$ 
is pointwise partially hyperbolic (with uniform constants) and center bunched for any $t\neq 0$. More precisely, for any $\epsilon>0$ sufficiently small, there exist $\rho\in (0,1)$ depending only on $g_t$, and a Riemannian structure on $TM$, such that 
for any unit vector ${\bf v}\in T_xM$, all $t>1$, $$\|D_xg_t({\bf v})\|<\rho^t,\;\text{if  }\;{\bf v}\in E^s(x);$$$$(1+\epsilon)^{-t}<\|D_xg_t({\bf v})\|<(1+\epsilon)^t,\;\text{if  }\;{\bf v}\in E^c(x);$$$$\rho^{-t}<\|D_xg_t({\bf v})\|,\;\text{if  }\;{\bf v}\in E^u(x).$$
Applying Theorem \ref{th-ce}, with respect to the above Riemannian metric, the time change (for large $T$) is pointwise partially hyperbolic and center bunched. The proof is complete.
\end{proof}

\section{Preparatory results on accessibility of time change}

\label{S-PA}

Assume that $g_t$ is a partially hyperbolic flow on a compact manifold $M$, and $\tau:M\to \R$ is a positive smooth function. Let $g_t^\tau$ be the time change.

\subsection{A correspondence of $su$-paths}

As in \cite{kk}, for $x\in M$ and $x'\in W^s(x)$, define $$PCF_{x,x'}(\phi)=\int_0^\infty(\phi(g_t(x))-\phi(g_t(x')))dt,$$and for $x\in M$ and $x'\in W^u(x)$, define $$PCF_{x,x'}(\phi)=\int_0^\infty(\phi(g_{-t}(x'))-\phi(g_{-t}(x)))dt.$$For an $su$ path $\mathcal C=\{x_0,x_1,\ldots,x_n\}$, here $x_{i+1}\in W^{s/u}(x_i)$ for $0\le i\le n-1$, let $$PCF(\mathcal C)(\phi)=\sum_{i=0}^{n-1}PCF_{x_i,x_{i+1}}(\phi).$$

\begin{proposition}\label{key-p-2}
For any $su$-path $\mathcal C=\{x_0,x_1,\ldots,x_n\}$ w.r.t. $g_t$, then $$\mathcal C'=\{x_0,g^\tau_{-t_1}x_1,\ldots,g^\tau_{-t_n}x_n\}$$ is an $su$ path w.r.t. $g_t^\tau$, where $\displaystyle t_k=\sum_{i=0}^{k-1}PCF_{x_i,x_{i+1}}(\tau)$ for $1\le k\le n$.
\end{proposition}
\begin{proof}
This is a direct generalization of Proposition \ref{inv-p}. By using the defined maps $\Phi^{s/u}_x$ in the proof of Proposition \ref{inv-p}, one can finish the proof by using induction. We omit the details here.
\end{proof}

A particular case is when $\mathcal C$ is a cycle ($x_0=x_n$). In this case, $$g^\tau_{-t_n}x_n=g^\tau_{-PCF(\mathcal C)}x_n,$$ this gives a hint of the importance of periodic cycle functional. We will use this observation in a crucial way later.

Let now$$\mathcal P_x:=\{\text{All }su\text{-paths }\;w.r.t.\; g_t \text{ starting at $x$}\};$$
$$\mathcal P^\tau_x:=\{\text{All }su\text{-paths }\;w.r.t.\; g^\tau_t \text{ starting at $x$}\}.$$ Proposition \ref{key-p-2} gives a correspondence of $su$-path w.r.t. $g_t$ and $su$-path w.r.t. $g_t^\tau$.

\begin{definition}\label{key-d-4}
Fix $x$, for any $su$-path $\mathcal C=\{x,x_1,\ldots,x_n\}\in \mathcal P_x$, define $\Phi_x:\mathcal P_x\to \mathcal P^\tau_x,$ $$\Phi_x(\mathcal C):=\{x,g^\tau_{-t_1}x_1,\ldots,g^\tau_{-t_n}x_n\}\in \mathcal P^\tau_x,$$where $\displaystyle t_k=\sum_{i=0}^{k-1}PCF_{x_i,x_{i+1}}(\tau)$ for $1\le k\le n$.
\end{definition}

\subsection{Sufficient conditions to accessibility}

We assume here that $g_t$ is accessible.

\begin{proposition}\label{key-p-5}
$g^\tau_t$ is accessible if for any $x\in M$, $t\in \mathbb R$, there is an $su$-path (w.r.t. $g^\tau_t$) connecting $x$ and $g_tx$.
\end{proposition}
\begin{proof}
Since $g_t$ is accessible, then for any $x,y\in M$, there is an $su$ path w.r.t. $g_t$ connects $x$ and $y$. WLOG, let $\mathcal C=\{x,x_1,\ldots, x_n,y\}$ be the path. By Proposition \ref{key-p-2} and Definition \ref{key-d-4}, $\Phi_x\mathcal C$ is an $su$-path w.r.t. $g_t^\tau$ that connects $x$ and $g^\tau_{-PCF(\mathcal C)}(y)$. Now by the assumption, there is an $su$ path $\mathcal C'$ w.r.t. $g_t^\tau$ connects $g^\tau_{-PCF(\mathcal C)}(y)$ and $y$. 
Combining $\Phi_x\mathcal C$ and $\mathcal C'$, it follows that $x$ and $y$ are in the same accessibility class of $g_t^\tau$. Since $x$ and $y$ are chosen arbitrarily, $g_t^\tau$ is accessible.
\end{proof}

By an orbit segment, we mean a set $\{g_s(y):s\in I\}$ for some $y\in M$ and a compact interval $I\subset\mathbb R$.

\begin{proposition}\label{key-p-6}
If for some $x\in M$, the accessibility class $AC(x)$ w.r.t. $g_t^\tau$ contains an orbit segment, then $g^\tau_t$ is accessible.
\end{proposition}
\begin{proof}
By the assumption, it is easy to see by translating along the orbit
that the accessibility class $AC(x)$  w.r.t. $g_t^\tau$ contains the whole orbit of $x$, namely $\{g_t(x):t\in\mathbb R\}\subset AC(x)$. For any $y\in M$, let $\mathcal C=\{y,x_1,\ldots, x_n,x\}$ be an $su$ path connects $y$ and $x$. Similar to the proof of Proposition \ref{key-p-5}, we have $y\in AC(x)$. Thus $g_t^\tau$ is accessible.
\end{proof}

\begin{lemma}\label{le-every}
For any $x$, the accessibility class $AC(x)$ w.r.t. $g_t^\tau$ has nonempty intersection with every orbit. 
\end{lemma}
\begin{proof}
This follows similarly from the proofs of previous two propositions.
\end{proof}

\begin{proposition}\label{key-p-a}
For any $x$, the accessibility class $AC(x)$ w.r.t. $g_t^\tau$ is either the whole manifold, or a codimension one topological sub-manifold. 
\end{proposition}
\begin{proof}
Observe that $g_t^\tau (AC(x))=AC(g_t^\tau(x))$ for any $x\in M$ and $t\in \mathbb R$. This means, in particular that for any $t\neq s$,
\begin{equation}\label{eq-dic}
 \text{ either }AC(g_t^\tau(x))=AC(g_s^\tau(x))\text{ or }AC(g_t^\tau(x))\cap AC(g_s^\tau(x))= \emptyset.
 \end{equation}Let $H(x)=\{t:g_t^\tau(x)\in AC(x)\}$. Then $H(x)$ is a group of $\R$, which is either discrete
(possibly dense) or equal to $\mathbb R$. We have that $H(x)$ is a group independent of $x$. Indeed, it is obvious to see that $H(x)$ is the same as long as $x$ lies in one orbit, then the claim follows from Lemma \ref{le-every} that it is also the same for different orbits.
 
 Let $\widetilde M$ be the universal cover of $M$, and for any $su$-path (w.r.t. $g_t$) on $M$, lift it to $\widetilde M$. Composing with $\Phi_x$, we obtain $su$-paths (w.r.t. $g_t^\tau$) on $\widetilde M$. This way, we can lift $AC(y)$ (w.r.t. $g_t^\tau$) to a set $\widetilde{AC}(y)$ in $\widetilde M$ satisfying that $y\in \widetilde{AC}(y)$. Notice that $\widetilde{AC}(y)$ is connected and $su$-path connected, and it is saturated by the lifted stable/unstable manifolds (w.r.t. $g_t^\tau$). Moreover, by Lemma \ref{le-every}, $\widetilde{AC}(y)$ has nonempty intersection with every lifted $g_t$ orbit on $\widetilde M$. 

Let $\tilde{H}(x)=\{t:g_t^\tau(x)\in \widetilde{AC}(x)\}$. Then $\tilde{H}(x)$ is a group of $\R$, which is either discrete (possibly dense) or equal to $\mathbb R$. Also, it is easy to see that $\tilde{H}(x)$ is a subgroup of $H(x)$. Now by \eqref{eq-dic}, similarly to the argument of $H(x)$, $\tilde{H}(x)$ is independent of the choice of $x$. Denote by $\tilde{H}$ 
the common value of $\tilde H(x)$
for all $x.$

Now consider the quotient space $\widetilde{AC}(x)/\langle g^\tau_t:t\in \tilde H\rangle$. It is a topological space with topology inducing from $\widetilde{AC}(x)$. Moreover, since $\widetilde{AC}(x)$ has nonempty intersection with any $g^\tau_t$ orbit, it follows that  $\widetilde{AC}(x)/\langle g^\tau_t:t\in \tilde H\rangle$ is homeomorphic to $\widetilde M/\langle g_t^\tau:t\in\R\rangle$. Therefore we can regard $\widetilde{AC}(x)$ as a $\tilde H$ bundle over $\widetilde M/\langle g_t^\tau:t\in\R\rangle$. Since $\widetilde{AC}(x)$ is connected, $\tilde H$ can not be discrete, which means that either $\tilde H=\{0\}$ or $\tilde H=\R$.

When $\tilde H=\R$, it implies that there is only one accessibility class, hence  $AC(x)$ is the whole manifold. If $\tilde H=\{0\}$ (in this case for any $x\in M$, $H(x)$ is discrete and countable), then $AC(x)$ is a codimension one topological manifold, as its lift $\widetilde{AC}(x)$ is homeomorphic to $\widetilde M/\langle g_t^\tau:t\in\R\rangle$.
\end{proof}

\begin{remark}
We believe, in the latter case of Proposition \ref{key-p-a}, one may establish the smoothness of the manifold, by the argument in \cite{W}.
\end{remark}

\begin{proposition}\label{key-p-7}
Let $\mathcal C$ be an $su$-path w.r.t. $g^\tau_t$ that connects $x$ to $g_t(x)$. Assume that $g_t(x)\neq x$. If $\mathcal C$ is homotopic to the orbit segment starting from $x$ to $g_t(x)$, then $g^\tau_t$ is accessible.
\end{proposition}
\begin{proof}
Let $\tilde{ \mathcal C}$ be the lifted path of $\mathcal C$ on the universal cover fixing $x$. Notice that by the homotopy assumption, $\tilde{\mathcal C}$ starts at $x$ and ends at $g_t(x)$. As in the proof of Proposition \ref{key-p-a}, this implies $\tilde H=\R$. Hence it follows that the whole manifold is one accessibility class. Therefore $g_t^\tau$ is accessible.
\end{proof}

\begin{proposition}\label{key-p-8}
If $g_t$ is accessible, and $g^\tau_t$ is not accessible, then $PCF(\tau)=0$ for all homotopically trivial $su$ cycles w.r.t. $g_t$. 
\end{proposition}
\begin{proof}
Assume there is a homotopically trivial $su$ cycle $\cC$, such that $PCF_\cC(\tau)\neq 0$.  Assume the cycle starts at $x$. Then after mapping $\cC$ by $\Phi_x$ to a $su$ path of $g^\tau_t$, it follows by Proposition \ref{key-p-2} that $x$ (the last point of $\cC$) will be mapped to $g^\tau_{-PCF_\cC(\tau)}(x)$. Therefore we obtain an $su$ path w.r.t. $g^\tau_t$ that satisfies the condition in Proposition \ref{key-p-7}. We then obtain the accessibility of $g^\tau_t$, contradicting to our assumption.
\end{proof}

The following is a generalization of the previous proposition, which plays an important role in our proof of accessibility.
\begin{proposition}\label{key-p-9}
Assume that $g_t$ is accessible, and $g^\tau_t$ is not accessible. Let $\mathcal C$ be an $su$ path w.r.t. $g_t$ that starts at $x$ and ends at $g_r(x)$ for some $r\ge 0$. Suppose $\mathcal C$ is homotopic to the orbit segment starting from $x$ to $g_r(x)$. Then $$PCF_{\mathcal C}(\tau)=\int_0^r\tau(g_t x)dt.$$
\end{proposition}
\begin{proof}
Arguing similarly to Proposition \ref{key-p-8}, we can map $\cC$ by $\Phi_x$ to obtain an $su$ path $\tilde\cC$ w.r.t. $g^\tau_t$. Here, by Proposition \ref{key-p-2}, $\tilde\cC$ starts at $x$ and ends at $g^\tau_{-PCF_\cC(\tau)}(g_rx)$. Let $R=\int_0^r\tau(g_t x)dt$. Then $g_r(x)=g^\tau_R(x)$. Hence $\tilde\cC$ ends at  $g^\tau_{-PCF_\cC(\tau)}(g_r(x))=g^\tau_{-PCF_\cC(\tau)+R}(x)$. Now by Proposition \ref{key-p-7}, we must have $-PCF_\cC(\tau)+R=0$, as desired.
\end{proof}

\section{Accessibility of smooth time changes}
\label{S-A}

\subsection{Time changes of homogeneous flows}
We assume $G$ is a (real or complex) semisimple Lie group without compact factors, $\Gamma$ is a cocompact lattice. Given an element $X\in$ Lie$(G)$, and assume it induces a partially hyperbolic flow $g_t$ (equivalently, $Ad(\exp(X))$ has at least one eigenvalue of modulus different from one.). Suppose $g_t$ is accessible, $m$ is the Haar measure. Let $\tau:G/\Gamma\to \mathbb R_+$ be a smooth function, and $\tau_0=\int_{G/\Gamma}\tau\ dm$. And denote $g_t^\tau$ the time change of $g_t$ by $\tau$. 

Let  $G^{s}$ ($G^u$ respectively) be the stable (unstable respectively) horospherical group of $g_t$, namely
$$G^{s}:=\{h\in G: \exp(tX)h\exp(-tX)\to Id,\;\text{as}\; t\to+\infty\},$$$$G^{u}:=\{h\in G:  \exp(tX)h\exp(-tX)\to Id,\;\text{as}\; t\to-\infty\}.$$

We are ready to prove Theorem \ref{main-a} in the following.

\begin{theorem}\label{new}
$g_t^\tau$ is accessible.
\end{theorem}
\begin{proof}
Firstly, notice that $g_t$ is locally accessible (Definition \ref{def-local-A}). Indeed, since $g_t$ is accessible, it implies that the stable bundle and unstable bundle together with their Lie brackets generate the whole tangent bundle. In particular, any two points in a small neighborhood can be connected by an $su$-path. Then it follows easily that $g_t$ is locally accessible. 

Following local accessibility, we have that there is a path $\cC$ w.r.t. $g_t$  satisfying the following properties:
\begin{itemize}
\item[(A)] it is not a cycle, 
\item[(B)] it starts at $x$ and ends at $g_r(x)$ for some $r> 0$,
\item[(C)] it is homotopic to the orbit segment starting from $x$ to $g_r(x)$.
\end{itemize}
Now assume that $\cC=\{x,h_1(x),h_2h_1(x),\cdots, h_nh_{n-1}\cdots h_1 (x)\}$ with $h_i\in G^{s/u}$. In particular, property (B) above means that $h_nh_{n-1}\cdots h_1=g_r$. Therefore, it is easy to see that for any $y$, $\cC_y:=\{y,h_1(y),h_2h_1(y),\cdots, h_nh_{n-1}\cdots h_1 (y)\}$ satisfies the properties (A), (B), (C).

Now assume $g^\tau_t$ is not accessible. Notice that $\cC_y$ satisfies the conditions in Proposition \ref{key-p-9}, for any $y$. Then by Proposition \ref{key-p-9}, it follows that for any $y$, $$PCF_{\mathcal C_y}(\tau)=\int_0^r\tau(g_t y)dt.$$Integrating the above for all $y$, we have \be\label{e-c-1}\int_{G/\Gamma}PCF_{\cC_y}(\tau)dm(y)=\int_{G/\Gamma}\int_0^r\tau(g_ty)dtdm(y).\ee
On one hand, it is easy to see from the definition of PCF and invariance of the measure $m$ under left translations, that \be\label{e-c-2}\int_{G/\Gamma}PCF_{\cC_y}(\tau)dm(y)=0.\ee
On the other hand, by switching the order of the integrals, we have
\be\label{e-c-3}
\int_{G/\Gamma}\int_0^r\tau(g_ty)dtdm(y)=\int_0^r\int_{G/\Gamma}\tau(g_ty)dm(y)dt=r\tau_0>0.\ee
Thus by \eqref{e-c-1}, \eqref{e-c-2} and \eqref{e-c-3}, we reach to a contradiction. Hence $g^\tau_t$ is accessible.
\end{proof}

\subsection{Fundamental group and accessibility}

If the dynamical system does not have homogeneous structure,
we can still have the following result, by using property of fundamental group of the manifold.
\begin{theorem}\label{t-5-1}
Let $N$ be a smooth manifold, and $\phi_t$ be a patially hyperbolic flow on $(N,\mu)$, with $\mu$ a smooth measure. Assume $\phi_t$ is accessible, and there is no nontrivial homomorphism from the fundamental group $\pi_1(N)$ to $\mathbb R$. Let $\tau:N\to\mathbb R_+$ be a smooth function, and denote $\phi_t^\tau$ the time change of $\phi_t$ by $\tau$. Then $\phi_t^\tau$ is accessible.
\end{theorem}
\begin{proof}
We prove it by contradiction. Assume that $\phi_t^\tau$ is not accessible.

For any point $x\in N$, let $\mathcal Cyc(x)$ be the set of all $su$-paths w.r.t. $\phi_t$ starting and ending at $x$. Let$$H(x):=\{u\in\mathbb R:  \exists\,\mathcal C\in \mathcal Cyc(x),\;s.t.\text{ $\Phi_x(\mathcal C)$ connects $x$ to }\phi^\tau_u(x)\}.$$ It is clear that $H(x)$ is a subgroup of $\mathbb R$.

Let $\mathcal C\in \mathcal Cyc(x)$ and assume $\mathcal C$ is homotopy trivial, then by Definition \ref{key-d-4}, $\Phi_x(\mathcal C)$ is an $su$ path w.r.t. $\phi_t^\tau$ and it connects $x$ with $\phi_a(x)$ for some $a\in\mathbb R$. Now by Proposition \ref{key-p-8}, $a=0$. Since by Proposition \ref{key-p-2}, the value $a$ is equal to the periodic cycle functional evaluated on the cycle $\mathcal C$, therefore the periodic cycle functional vanishes over all homotopy trivial cycles. 

With this observation, it follows that there is a homomorphism from the fundamental group $\pi_1(N)$ to $H(x)\subset \mathbb R$. By the assumption on $\pi_1(N)$, $H(x)$ must be trivial. This means that the periodic cycle functional vanishes over all cycles. Hence by \cite{W}, $\tau-\int\tau$ is a (smooth) coboundary with respect to the vector field generating $\phi_t$. This in turn gives that $\phi_t^\tau$ is smoothly conjugate to $\phi_t$, up to a constant time change. By this conjugacy, $\phi_t^\tau$ must be accessible since $\phi_t$ is. This is a contradiction to the assumption we start with.
\end{proof}

We can apply Theorem \ref{t-5-1} to certain homogeneous flows.
We make the following assumptions on the homogeneous space: 
\begin{itemize}
\item[(i)] $G$ is a connected semisimple Lie group without compact factors, 
\item[(ii)] The real rank of $G$ is greater than one,
\item[(iii)] The fundamental group of $G$ is finite cyclic,
\item[(iv)] $\Gamma<G$ is an irreducible cocompact lattice.
\end{itemize}

Under these assumptions, $G$ does not have any factor from the following list
\begin{itemize}
\item $SO(n,2)$ $(n\ge 1)$, $Sp(2n,\mathbb R)$ $(n\ge 2)$ (The fundamental group contains $\mathbb Z$.)
\end{itemize}

Let $M=G/\Gamma$. Given an element $X\in$ Lie$(G)$, and assume it induces a partially hyperbolic flow $g_t$. It follows from \cite{ps} that $g_t$ is accessible if and only if $Ad(\exp(X))$ restricted to each factor of $G$ has at least one eigenvalue of modulus different from one. Let $\tau:M\to \mathbb R_+$ be a smooth function, and denote $g_t^\tau$ the time change of $g_t$ by $\tau$. Our main result in this section is the following.

\begin{theorem}
$g_t^\tau$ is accessible.
\end{theorem}
\begin{proof}
This follows from Theorem \ref{t-5-1}.  It suffices to check the condition on the fundamental group of $G/\Gamma$. By assumption (iii), $\pi_1(G/\Gamma)$ is a finite extension of $\Gamma$. But now by assumption (ii) and Margulis normal subgroup theorem \cite{ma-1},\cite[Theorem 4', Introduction]{ma-2}, there is no nontrivial homomorphism from $\Gamma$ to $\mathbb R$. Hence there is no nontrivial homomorphism from $\pi_1(G/\Gamma)$ to $\mathbb R$.
\end{proof}

\section{Stable ergodicity}

\label{S-C}

In our setting, it is known that stable accessibility implies stable ergodicity (by Theorem \ref{bw}). So we just need to show
\begin{theorem}
Under the assumptions of Theorem \ref{main}, the time change $g_t^\tau$ is stably accessible.
\end{theorem}

In order to prove stable accessibility, it is convenient to use the notion of centre engulfing, which was introduced in \cite{BPW,BW-1}.

Given a partially hyperbolic diffeomorphism $f: N \to N$, let $d$ be the
dimension of the centre bundle $E^c$. The diffeomorphism $f : N \to N$ is {\it centre engulfing} at a point $x \in N$ if:
\begin{itemize}
\item[({\bf a1})] there is a continuous family of $su$-paths $\Psi_z : [0,1] \to N$ indexed by points $z \in Z$, where $Z$ is a compact $d$-dimensional manifold with boundary;
\item[({\bf a2})] the number of legs in $\Psi_z$ is uniformly bounded;
\item[({\bf a3})] each path $\Psi_z$ begins at $x$ and ends in the centre manifold $W^c(x)$ through $x$ (i.e. $\Psi_z(0) = x$ and $\Psi_z(1) \in W^c(x)$ for all $z \in Z$);
\item[({\bf a4})] the paths $\Psi_z$ corresponding to $z \in \partial Z$ do not end at $x$; and finally
\item[({\bf a5})] the map $ (Z,\partial Z) \to (W^c(x),W^c(x) -\{x\})$ defined by $z\mapsto \Psi_z(1)$ wraps around $x$ non-trivially (i.e. the map has non-zero degree in the sense that the induced map on homology $H_d(Z, \partial Z, \mathbb Z) = \mathbb Z$ to $H_d(W^c(x), W^c(x) -\{x\}, \mathbb Z) = \mathbb Z $ is nonzero).
\end{itemize}
Observe that these five conditions are open under $C^1$-small perturbations of $f$. Moreover, centre engulfing with accessibility implies stable accessibility as stated below.

\begin{lemma}[\cite{BW-1}, Corollary 5.3]
If $f : N \to N$ is centrally engulfing at one point and it is possible to join any point of $N$ to $x$ by an $su$-path, then $f$ is stably accessible.
\end{lemma}

\subsection{Centre engulfing}

This subsection is devoted to the proof of centre engulfing for smooth time changes. To start with, we first notice that
the stable accessibility (and centre engulfing) for the algebraic flow $g_t$ is proven in \cite{BP,ps1,ps}. We make the statement in a precise way as follows.

\begin{proposition}[\cite{BP, ps1,ps}]\label{pp-s}
If $g_t$ is accessible, then it is centre engulfing, and hence stably accessible.
\end{proposition}

\begin{remark}\label{re-s}
In fact, the arguments in \cite{BP, ps1,ps} (e.g. \cite[Theorem 3.2]{ps1}) yield to a much stronger result than centre engulfing. Namely, in addition to the properties of centre engulfing, the map $\Psi_{(\cdot)}(1):Z\to W^c(x)$ can be chosen to be one to one, and smooth, and the path $\Psi_z(\cdot):[0,1]\to N$ is homotopic to the segment (inside $W^c(x)$) connects $x$ to $\Psi_z(1)$ for any $z\in Z$. Moreover, the image of $Z$ can be restricted in a sufficiently small ball around $x$. 
\end{remark}

For any two $su$ paths $\mathcal C_1=\{x,\cdots,y\}$ and $\mathcal C_2=\{y,\cdots,z\}$, define $\star$ as follows: $$\mathcal C_1\star\mathcal C_2=\{x,\cdots, y,\cdots, z\}.$$If $\mathcal C=\{x_0,x_1,\cdots,x_n\}$, then denote $[\mathcal C]^{-1}:=\{x_n,x_{n-1},\cdots,x_0\}$.

We are now ready to prove

\begin{proposition}
The time change $g_t^\tau$ is centre engulfing.
\end{proposition}
\begin{proof}
Notice first that $g_t$ and its time change $g_t^\tau$ share the same center foliation. Fix a point $x\in M$ such that $g_t(x)\neq x$ for $t\neq 0$. By Proposition \ref{pp-s}, it follows that $g_t$ is centre engulfing at $x$. More precisely, there exists a continuous family of $su$-paths $\Psi_z : [0,1] \to M$ indexed by points $z \in Z$, where $Z$ is a compact $d$-dimensional manifold with boundary, satisfying ({\bf a2}), ({\bf a3}), ({\bf a4}), and ({\bf a5}). And by Remark \ref{re-s}, we assume that the image $\Psi_{Z}(1)\subset W^c(x)$ is a ball of radius $r$, for $r>0$ small, and  the path $\Psi_z(\cdot):[0,1]\to N$ is homotopic to the segment (inside $W^c(x)$) connects $x$ to $\Psi_z(1)$ for any $z\in Z$.
A geometric picture of the map $\Psi_{(\cdot)}(1)$ is provided below as Figure \ref{Z-1}.

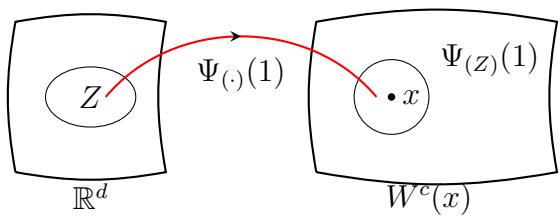
\begin{figure}[h]
\begin{center}
\begin{tikzpicture}

\coordinate (A) at (0,0);
\coordinate (B) at (2,0);
\coordinate (C) at (2,2);
\coordinate (D) at (0,2);

\coordinate (E) at (4,0);
\coordinate (F) at (7.2,0);
\coordinate (G) at (7.2,2);
\coordinate (H) at (4,2);
\coordinate (I) at (5,1);
\coordinate (J) at (1,1);
\node at (1,0) [below] {$\R^d$};
\node at (5.5,0) [below] {$W^c(x)$};
\node at (I) [circle, fill ,inner sep=1pt] {};
\node at (I) [draw, circle,inner sep=10pt] {};
\node at (I) [right] {$x$};
\node at (5.5,1.5) [right] {$\Psi_{(Z)}(1)$};
\node at (J) [] {$Z$};
\node at (3,1.3) [] {$\Psi_{(\cdot)}(1)$};

\draw (1,1) ellipse (.6cm and .4cm);
\draw[thick,black] (A) to [bend right=10] (B) to [bend left=10] (C) to [bend right=10] (D)to [bend right=10] (A);

\draw[thick,black] (E) to [bend right=10] (F) to [bend left=10] (G) to [bend right=10] (H)to [bend right=10] (E);

\draw[thick, red, postaction={on each segment={mid arrow=black}}] (1.2,1) to [bend left=50] (4.8,1);

\end{tikzpicture}

\caption{The map $\Psi_{(\cdot)}(1)$}
\label{Z-1}
\end{center}\end{figure}

Notice that the assumptions above enable us to obtain a family of maps $\Psi$ starting at any point $y\in M$. Namely, for any $z\in Z$, $\Psi_z$ represents a path $$\cC_x(z)=\{x,h_1(x),h_2h_1(x),\cdots, h_nh_{n-1}\cdots h_1 (x)\}$$ with $h_i\in G^{s/u}$, then for any $y$, define $\cC_y:=\{y,h_1(y),h_2h_1(y),\cdots, h_nh_{n-1}\cdots h_1 (y)\}$ (we denote this map as $\Psi_Z^y:[0,1]\to W^c(y)$), thus we immediately have $g_t$ is centre engulfing at $y$. Here, the homotopy assumption is important.

For any $x$, we consider $\Phi_x\circ\Psi_z^x$ indexed by $z\in Z$, where $\Phi_x$ is the map from Definition \ref{key-d-4}. Since $\Phi_x$ maps continuously $su$-paths w.r.t. $g_t$ to $su$-paths w.r.t. $g_t^\tau$, it follows that indexed by $Z$, $\Phi_x\circ\Psi_z^x: [0,1]\to M$ is a continuous family of $su$-paths w.r.t. $g_t^\tau$. It is straightforward that this map $\Phi_x\circ\Psi_z^x$ satisfies ({\bf a1}), ({\bf a2}) and ({\bf a3}), however there are problems for ({\bf a4}) and ({\bf a5}). This is due to the fact that the flow orbit lies in the center manifold and the map $\Phi_x$ acts essentially by sliding paths along the orbits. On one hand, it is possible that the set $\{\Phi_x\circ\Psi_z^x(1):z\in \partial Z\}$ contains $x$, which violates ({\bf a4}). On the other hand, it may be even worse that the set $\{\Phi_x\circ\Psi_z^x(1):z\in  Z\}$ does not contain $x$, certainly violating ({\bf a5}). To overcome these two ``bad" situations (if needed), we will add a few legs in the path $\Phi_x\circ\Psi_z^x$ by using the operation $\star$.

We first focus on the orbit of $x$.
Assume $t_1<0$ and $t_2>0$ are such that
$$\{g_t(x):t_1\le t\le t_2\}\cap \{\Psi_z^x(1):z\in\partial Z\}=\{g_{t_1}(x),g_{t_2}(x)\}.$$ Both $t_1$ and $t_2$ are finite, due to the fact that $\Psi_z^x$ satisfies ({\bf a4}) and ({\bf a5}). By compactness of $Z$, it follows that $ \{g_t(x):t_1< t< t_2\}\subset \{\Psi_z^x(1):z\in Z-\partial Z\}$.

Let $z_i\in \partial Z$ be such that the path $\Psi_{z_i}^x$ ends at $g_{t_i}(x)$, for $i=1,2$. By the definition of $\Phi_x$, we have that $$\Phi_x\circ\Psi_{z_1}^x(1)=g^\tau_{-PCF_{\Phi_x\circ\Psi_{z_1}^x}-\int_{t_1}^0\tau(g_t(x))dt}(x),\text{  }\Phi_x\circ\Psi_{z_2}^x(1)=g^\tau_{-PCF_{\Phi_x\circ\Psi_{z_2}^x}+\int^{t_2}_0\tau(g_t(x))dt}(x).$$
Since by our assumption on the map $\Psi_{Z}^{(\cdot)}$ (which follows from Remark \ref{re-s}), the path $\Psi_{z_i}^{y}$ is homotopic to the orbit segment from $y$ to $g_{t_i}(y)$. Arguing similarly to the proof of Theorem \ref{new}, we obtain that,
there is a point $x$ such that \be\label{eq-neq}-PCF_{\Phi_x\circ\Psi_{z_1}^x}-\int_{t_1}^0\tau(g_t(x))dt< -PCF_{\Phi_x\circ\Psi_{z_2}^x}+\int^{t_2}_0\tau(g_t(x))dt.\ee

From now on, we fix the $x$ satisfying \eqref{eq-neq}, and set $L_x$ ($R_x$ respectively) to be the value of the left (right respectively) hand side expression in \eqref{eq-neq}.
Since the map $\Phi_x\circ\Psi_{z_1}^x$ is continuous,  then the (nontrivial) orbit segment $\mathcal{O}_1:=\{g^\tau_t(x): t\in [L_x,R_x]\}$ is contained in $\{\Phi_x\circ\Psi_{z}^x(1):z\in Z\}$. Let $\mathcal{O}_2:=\{g^\tau_t(x): t\in (L_x,R_x)\}\subset \mathcal O_1$.

There are two cases from here. If $x\in \mathcal O_2$, then the family of continuous maps $\Phi_x\circ\Psi_z^x:[0,1]\to M$  is what we want. Indeed, it is straightforward that ({\bf a1}), ({\bf a2}), ({\bf a3}) and ({\bf a4}) are all satisfied, as for ({\bf a5}), the degree of the induced map does not change, since here $\Phi_x$ varies continuously, and as a result the induced map of   $H_d(Z, \partial Z, \mathbb Z) = \mathbb Z$ to $H_d(W^c(x), W^c(x) -\{x\}, \mathbb Z) = \mathbb Z $ remains the same. If otherwise $x\notin \mathcal O_2$, then pick any $y\in \mathcal O_2$, since $g_t^\tau$ is accessible by Theorem \ref{new}, there is a path $\cC$ starts at $y$ and ends at $x$, then we argue similarly as the previous case that the family of continuous maps $\tilde{\Psi}_z$ defined by $$\tilde{\Psi}_z(t):=\begin{cases}\vp(t),\text{ if }t\in[0,1/2];\\ \cC\star[\Phi_x\circ\Psi_z^x(2t-1)],\text{ if }t\in [1/2,1]\end{cases}$$suffices for our purpose. Here $\vp(t)$ is a continuous parametrization of the path $\cC$ satisfying $\vp(0)=y$ and $\vp(1/2)=x$. We remark that in this case the paths defined by these maps starts at $y$ instead of $x$. Therefore, we obtain centre engulfing at the point $x$ or $y$, completing the proof.
\end{proof}

\text{ }\\
Department of Mathematics\\University of Maryland College Park\\
4176 Campus Drive - William E. Kirwan Hall\\
College Park, MD 20742-4015\\
\url{dongchg@umd.edu}

\end{document}